\documentclass[final,nomarks]{dmtcs-episciences}

\usepackage[table]{xcolor}
\usepackage[utf8]{inputenc}
\usepackage{subfigure}

\usepackage[T1]{fontenc}
\usepackage{amsmath,amssymb,amsthm}
\usepackage{thmtools}
\usepackage{cleveref}
\usepackage{hyperref}
\usepackage{url}
\usepackage{subcaption}
\usepackage{enumerate}
\usepackage{diagbox}

\def\paren#1{\left( #1 \right)}
\def\acc#1{\left\{ #1 \right\}}

\renewcommand{\le}{\leqslant}
\renewcommand{\ge}{\geqslant}

\theoremstyle{plain}
\newtheorem{theorem}{Theorem}
\newtheorem{lemma}[theorem]{Lemma}

\theoremstyle{definition}

\theoremstyle{remark}
\newtheorem{remark}[theorem]{Remark}

\title{Words avoiding the morphic images of most of their factors}
\author{Pascal Ochem \and Matthieu Rosenfeld\thanks{This work is supported by the ANR project CADO (ANR-24-CE48-3758-01).}}
\affiliation{
  LIRMM, CNRS, Universit\'e de Montpellier, Montpellier, France.}

\begin{document}
\publicationdata{vol. 27:3}{2025}{13}{10.46298/dmtcs.15919}{2025-06-23; 2025-06-23; 2025-09-08; 2025-09-14}{2025-09-17}

\maketitle

\vspace{0mm}

\begin{abstract}
We say that a finite factor $f$ of a word $w$ is \emph{imaged} if there exists a non-erasing morphism $m$,
distinct from the identity, such that $w$ contains $m(f)$.
We show that every infinite word contains an imaged factor of length at least 6 and that 6 is best possible.
We show that every infinite binary word contains at least 36 distinct imaged factors and that 36 is best possible.
\end{abstract}

\section{Introduction}\label{sec:intro}
Following the construction by Fraenkel and Simpson~\cite{Fraenkel(1995)} of an infinite binary word containing only
the squares \texttt{00}, \texttt{11}, and \texttt{0101}, some results in the literature have shown that
infinite binary words can also contain finitely many distinct regularities that are more general 
than squares. In particular, the following result considers generalizations of squares such that one half of the pseudosquare
is allowed to be a morphic image of the other half, for some non-erasing morphism.

\begin{theorem}~\cite{Ng(2019)}\label{thm10prime}
There exists an infinite binary word that avoids all factors of the form $fm(f)$ and $m(f)f$,
for all non-erasing binary morphisms $m$, with $|f|\ge5$.
\end{theorem}

Notice that classical squares are indeed forbidden in Theorem~\ref{thm10prime} when $m$ is the identity.
Notice also that $m$ is required to be non-erasing to prevent that $m(f)$ is empty.

In this paper, we investigate whether we can remove the constraint that $f$ and $m(f)$ are adjacent.
That is, does there exist an infinite word $w$ such that $f$ and $m(f)$ are both factors of $w$ only for finitely many $f$'s?
The answer is obviously negative because $m$ can be the identity, that is, the morphism such that
$m(i)=i$ for every letter $i$ in the alphabet of $w$.
This implies that for every $n\ge1$, the prefix $p_n$ of length~$n$ of $w$ and the word $m(p_n) = p_n$ are both factors of $w$.
To get rid of such a triviality, we require that $m$ is not the identity.
Thus, we say that a morphism is \emph{admissible} if it is non-erasing and distinct from the identity.
Then we say that a finite factor $f$ of a word $w$ is \emph{imaged}
if there exists an admissible morphism $m$ such that $w$ contains $m(f)$. 

\begin{remark}{\ }\label{rem}
\begin{itemize}
\item[(i)] If $f$ is imaged in $w$, then every factor of $f$ is imaged in $w$.
\item[(ii)] If $f$ is a factor of $w$ that does not contain every letter of $w$, then $f$ is imaged in $w$.
\item[(iii)] In particular, the empty word $\varepsilon$ is imaged in every infinite word.
\end{itemize}    
\end{remark}

To prove \Cref{rem}.(ii), let $f$ be a factor of $w$ that does not contain the letter~$a$ of $w$.
Let $m$ be the admissible morphism such that $m(a)=aa$ and $m(i)=i$ for every letter $i$ of $w$ other than $a$.
Then $m(f)=f$ even if $m$ is not the identity. So $f$ is imaged since $w$ contains $f$ and $m(f)$.
% \begin{proof}
% Consider a word $w$ in the proof of Theorem~\ref{yes34}.
% Notice that $w$ contains \texttt{000011} and that $w$ contains no 4-power other than \texttt{0000}.
% \begin{itemize}
% \item So \texttt{00001} is imaged in $w$ because $m(\texttt{00001})=\texttt{000011}$ (with $m: \texttt{0}\mapsto\texttt{0}, \texttt{1}\mapsto\texttt{11}$) is a factor of $w$.
% However, the factor \texttt{0000} of \texttt{00001} is not imaged in $w$,
% since $w$ contains no 4-power other than \texttt{0000}.
% \item Just notice that $m: \texttt{0}\mapsto\texttt{0}, \texttt{1}\mapsto\texttt{11}$
% is not the identity, whereas $f=m(f)$ for $f=\texttt{0000}$.\qedhere
% \end{itemize}
% \end{proof}

Our results consider the length and the number of imaged factors that are unavoidable in infinite binary words.
First, we prove that long imaged factors can be avoided.
\begin{theorem}\label{yes7}
There exist exponentially many binary words that avoid imaged factors of length 7.
\end{theorem}
Then we prove that the length 7 is sharp in the previous result.
\begin{theorem}\label{not6}
Every infinite word over a finite alphabet contains an imaged factor of length 6.
\end{theorem}
This naturally leads to wonder what is the minimum number of imaged factor in an infinite binary word. We also obtain a precise answer in this case.
\begin{theorem}\label{yes36}
There exist exponentially many binary words that contain at most 36 imaged factors.
\end{theorem}

\begin{theorem}\label{no35}
Every infinite binary word contains at least 36 imaged factors.
\end{theorem}

Now let us give useful standard definitions.
A \emph{repetition} $w$ is a finite prefix of $u^\omega$ such that $u$ is a non-empty finite word and $|w|>|u|$.
The \emph{period} of $w$ is $|u|$ and the \emph{exponent} of $w$ is $\tfrac{|w|}{|u|}$.
A word is $(\beta^+,n)$-free if it contains no repetition $u^e$
with $|u|\ge n$ and $e>\beta$. A word is $\beta^+$-free if it is $(\beta^+,1)$-free.
A morphism~$f:\Sigma^*\rightarrow\Delta^*$ is \emph{$q$-uniform} if $|f(a)|=q$ for every $a\in\Sigma$, and $f$ is \emph{synchronizing}
if for all $a,b,c\in\Sigma$ and $u,v\in \Delta^*$, such that $f(ab)=uf(c)v$, then either $u=\varepsilon$ and $a=c$, or $v=\varepsilon$ and $b=c$.

The proofs of~\Cref{yes7,yes36} use the following lemma.

\begin{lemma}\label{sync}~\cite{Ochem(2006)}
Let $\alpha,\beta\in\mathbb{Q},\ 1<\alpha<\beta<2$ and $n\in\mathbb{N}^*$.
Let $h:\Sigma^*_s\rightarrow\Sigma^*_e$ be a synchronizing $q$-uniform morphism (with $q\ge 1$).
If $h(w)$ is $\paren{\beta^+,n}$-free for every $\alpha^+$-free word $w$ such that
$|w|<\max\paren{\frac{2\beta}{\beta-\alpha},\frac{2(q-1)(2\beta-1)}{q(\beta-1)}}$,
then $h(t)$ is $\paren{\beta^+,n}$-free for every (finite or infinite) $\alpha^+$-free word $t$.
\end{lemma}
To obtain~\Cref{yes7,yes36}, we use Lemma~\ref{sync} such that the pre-images are ternary $\tfrac74^+$-free words.
Since we know that there exist exponentially many ternary $\tfrac74^+$-free words~\cite{Ochem(2006)},
this proves that there exist indeed exponentially many of the binary words considered in~\Cref{yes7,yes36}.

We provide in the associated git repository a program that verifies if Lemma~\ref{sync} can be applied to a given morphism \cite{giteRepo(2025)}.
 
\section{Avoiding imaged factors of length 7}\label{s7}
Now we prove \Cref{yes7}.
Let $w$ be the image of any ternary $\tfrac74^+$-free word by the following 37-uniform morphism.
\begin{align*}
\texttt{0}&\rightarrow\texttt{0001110101001100011101000110101001101}\\
\texttt{1}&\rightarrow\texttt{0001110101000110100110001110100110101}\\
\texttt{2}&\rightarrow\texttt{0001110100110001110101000110101001101}
\end{align*}
This morphism was found using the method described in~\cite{Ochem(2006)}.
The following properties of $w$ can be checked by standard techniques and~\Cref{sync}.
\begin{itemize}%[label=(\alph*)]
\item[(a)] $w$ contains no factor in $F=\acc{\texttt{0000},\texttt{1111},\texttt{0010},\texttt{1011},\texttt{010101}}$.\label{itemF}
\item[(b)] $w$ is $\paren{\tfrac{289}{148}^+, 3}$-free and contains no square other than \texttt{00}, \texttt{11}, \texttt{0101}, and \texttt{1010}.\label{itemSQo}
\item[(c)] $w$ does not contain both a factor $f$ of length 7 and $\overline{f}$
(where $\overline{f}$ is the bit-complement of $f$, that is, the image of $f$ by $\texttt{0}\rightarrow\texttt{1}$; $\texttt{1}\rightarrow\texttt{0}$).\label{itemC}
\item[(d)] Every factor $f$ of length 7 of $w$ contains at least one of the following:\label{itemSQi}
\begin{itemize}
\item the factor \texttt{0101},
\item the factor \texttt{1010},
\item both the factors \texttt{00} and \texttt{11}.
\end{itemize}
\end{itemize}
Let us show that $w$ does not contain two distinct factors $f$ and $m(f)$ such that $|f|=7$ and $m$ is a non-erasing morphism.
To write a morphism $m$, we use the notation $m(\texttt{0})/m(\texttt{1})$.

By (d), both letters \texttt{0} and \texttt{1} are contained in a square of $f$.
The $m$-images of these squares are constrained by (b), so that
$|m(\texttt{0})|\le2$ and $|m(\texttt{1})|\le2$.
Here are further restrictions on~$m$:
\begin{itemize}
\item The identity morphism $\texttt{0}/\texttt{1}$ is ruled out by definition.
\item The bit-complement morphism $\texttt{1}/\texttt{0}$ is ruled out by (c).
\item $m(\texttt{0})=m(\texttt{1})$ is ruled out since otherwise $m(f)=\paren{m(\texttt{0})}^7$ would contain
\texttt{0000}, \texttt{1111}, or \texttt{010101}, contradicting (a)).
\end{itemize}
In particular, $|m(\texttt{0})|+|m(\texttt{1})|\ge3$. So $f$ contains neither \texttt{0101} nor \texttt{1010},
since otherwise $m(f)$ would contain a square with period at least 3, contradicting (b).
By (d), $f$ contains both \texttt{00} and \texttt{11}.
This implies one more restriction on~$m$:
\begin{itemize}
\item $\acc{m(\texttt{0}),m(\texttt{1})}\cap\acc{\texttt{00},\texttt{11}}=\emptyset$ since otherwise $m(f)$ would contain \texttt{0000} or \texttt{1111}.
\end{itemize}
 
So $m$ is in $S_m=\{
\texttt{0}/\texttt{01},\ 
\texttt{0}/\texttt{10},\ 
\texttt{1}/\texttt{01},\ 
\texttt{1}/\texttt{10},\ 
\texttt{01}/\texttt{0},\ 
\texttt{01}/\texttt{1},\ 
\texttt{10}/\texttt{0},\ 
\texttt{10}/\texttt{1},\ 
\texttt{01}/\texttt{10},\ 
\texttt{10}/\texttt{01}\}$
and $f$ is in $S_f=\{
\texttt{0001101},
\texttt{0001110},
\texttt{0011000},
\texttt{0011101},
\texttt{0100011},
\texttt{0100110},
\texttt{0110001},
\texttt{1000110},\\
\texttt{1000111},
\texttt{1001100},
\texttt{1001101},
\texttt{1100011}\}$.
Finally, we check that for every $m\in S_m$ and every $f\in S_f$,
the image $m(f)$ contains a factor in $F$, and thus is not a factor of $w$.

% ADDENDUM

% The image of the fixed point $b_3$ of $\texttt{012}/\texttt{02}/\texttt{1}$ by
% \begin{align*}
% \texttt{0}&\rightarrow\texttt{11010100110100011101001101010011101000}\\
% \texttt{1}&\rightarrow\texttt{11010100110100011101010011101000}\\
% \texttt{2}&\rightarrow\texttt{1101001}
% \end{align*}
 
% may be a better candidate w.r.t. the number of imaged factors.

\section{Imaged factors of length 6 are unavoidable}\label{s6}
This section is devoted to \Cref{not6}.
Suppose that \Cref{not6} is false, that is, there exists an infinite word $w$ over a finite alphabet that 
does not contain two factors $f$ and $m(f)$ such that $|f|=6$ and $m$ is an admissible morphism.
Notice that the conditions on $f$ and $m(f)$ define a factorial language.
By a standard argument, we can assume that $w$ is recurrent, that is, all its finite factors
appear infinitely often.

\begin{lemma}\label{twice}
Every factor of length 6 of $w$ contains every letter of $w$ at least twice.
\end{lemma}

\begin{proof}
By remark~\ref{rem}.(ii), every factor of length 6 of $w$ contains every letter of $w$ at least once.
Now, we rule out that a factor $f$ of length 6 of $w$ contains a letter $a$ exactly once.
We write $f=pas$ where $p$ and $s$ are possibly empty words that do not contain $a$.
Since $w$ is recurrent, it contains a factor $fvf=pasvpas$ for some word $v$.
Notice that $pasvpas=m(pas)$ where $m$ is the admissible morphism
such that $m(a)=asvpa$ and $m(i)=i$ for every other letter $i$.
This is a contradiction since $w$ contains both $f$ and $m(f)$.
\end{proof}

Now we consider the size $s$ of the alphabet of $w$. \Cref{twice} implies $s\le3$.
We can rule out $s=1$ since $\texttt{0}^{12}$ is the image of $\texttt{0}^6$
by the  admissible morphism $\texttt{0}\mapsto\texttt{00}$.
If $s=3$, then \Cref{twice} implies that every factor of $w$ of length 6
contains each of the three letters exactly twice.
To ensure this property of the factors $w_kw_{k+1}w_{k+2}w_{k+3}w_{k+4}w_{k+5}$
and $w_{k+1}w_{k+2}w_{k+3}w_{k+4}w_{k+5}w_{k+6}$, we must have $w_k=w_{k+6}$.
So $w$ is periodic with period 6. Let $f$ be any factor of $w$ of length~6.
Then $w$ contains $f^6=m(f)$, where $m$ is the admissible morphism such that
$m(\texttt{0})=m(\texttt{1})=m(\texttt{2})=f$.
This contradiction rules out $s=3$.

There remains the case $s=2$, that is, $w$ is a binary word.
Let us show that $w$ does not contain \texttt{0000}.
By \Cref{twice}, \texttt{0000} can only extend to \texttt{11000011}.
However $\texttt{11000011}=m(\texttt{100001})$ with $m(\texttt{0})=\texttt{0}$ and $m(\texttt{1})=\texttt{11}$.
So $w$ avoids \texttt{0000}, and by symmetry, $w$ avoids \texttt{1111}.
Now, we describe the computation that checks that no such infinite binary word $w$ exists.
It is a backtracking algorithm which backtracks in the following situations:
\begin{itemize}
\item The suffix of the current word is \texttt{0000} or \texttt{1111}.
\item The current word contains the complement $\overline{s}$ of its suffix $s$ of length 6.
\item For $f\in\{\texttt{010101},\texttt{001100},\texttt{001001},\texttt{011011},\texttt{001010},\texttt{010100},\texttt{011101},\texttt{010001},\\
\texttt{011100},\texttt{001110},\texttt{011000},\texttt{000110},\texttt{010111},\texttt{000101},\texttt{010110},\texttt{011010}\}$,
the current word contains either $f$ or $\overline{f}$, and also contains as a suffix an image $m(f)$ such that $|m(\texttt{0})|+|m(\texttt{1})|\ge3$.
\end{itemize}
This backtrack finishes\footnote{The program we used is available in the associated git repository~\cite{giteRepo(2025)}.},
which implies that every binary word contains an imaged factor of length 6. This concludes the proof of~\Cref{not6}.

\section{Binary words with only 36 imaged factors}
Let us prove~\Cref{yes36}.
Let $w$ be the image of any ternary $\tfrac74^+$-free word by the following 342-uniform morphism.
\begin{align*}
&\texttt{0}\rightarrow p\texttt{1100011000011000011000110001100001100011000110000110000110001}\\
&\texttt{100011000011000110001100001100011000110000110000110001100011000011}\\
&\texttt{000110001100001100001100011000110000110001100011000011000011000110}\\
&\texttt{00110000110000110001100011000011000110001100001100001100011000110}\\
\end{align*}
\begin{align*}
&\texttt{1}\rightarrow p\texttt{0110001100011000011000011000110001100001100011000110000110000}\\
&\texttt{110001100011000011000110001100001100011000110000110000110001100011}\\
&\texttt{000011000110001100001100001100011000110000110001100011000011000110}\\
&\texttt{00110000110000110001100011000011000110001100001100001100011000110}\\
&\texttt{2}\rightarrow p\texttt{0110001100011000011000011000110001100001100011000110000110000}\\
&\texttt{110001100011000011000011000110001100001100011000110000110000110001}\\
&\texttt{100011000011000110001100001100001100011000110000110001100011000011}\\
&\texttt{00011000110000110000110001100011000011000110001100001100001100011}
\end{align*}
where $p=\texttt{000110000110001100011000011000110001100001100001100011000}$\\
\texttt{110000110001100011000011000}.

The following properties can be checked by standard techniques and~\Cref{sync}.
\begin{itemize}%[label=(\alph*)]
\item[(a)] $w$ avoids the factors in $F=\acc{\texttt{010},\texttt{101},\texttt{111},\texttt{1001},\texttt{00000}}$.\label{ff}
%\item $h(\texttt{0})$, $h(\texttt{1})$, and $h(\texttt{2})$ are palindromes.\label{pal}
\item[(b)] $w$ is $\paren{\tfrac{1321}{684}^+, 245}$-free.\label{2p}
%\item The only 4-power contained in $w$ is \texttt{0000}.\label{4p}
%\item Every factor $f$ of length 7 of $h(b_3)$ contains both the factors \texttt{000} and \texttt{11}.
%\item $h(b_3)$ belongs to $\{a_4a_3a_3,a_4a_4a_3a_3\}^*$ with $a_3=11000$ and $a_4=110000$ \label{decomposition} 
\end{itemize}
First, we consider the set $T=\{
\texttt{0000110},
\texttt{0110000},
\texttt{0110001},
\texttt{1000011},
\texttt{1000110},
\texttt{1100001},\\
\texttt{00011000}\}$ of factors of $w$ and we show that these factors are not imaged in $w$.
Suppose for contradiction that $f$ is a word in $T$ and that $m$ is an admissible morphism such that $m(f)$ is a factor of $w$.
Notice that \texttt{00} and \texttt{11} are factors of $f$, so that $|m(\texttt{0})|\le244$ and $|m(\texttt{1})|\le244$ by~(b).
Then a computer check shows that $f$ is not imaged in $w$.

By~(a), the factors in $T'=F\cup T=\{
\texttt{010},
\texttt{101},
\texttt{111},
\texttt{1001},
\texttt{00000},
\texttt{0000110},
\texttt{0110000},\\
\texttt{0110001},
\texttt{1000011},
\texttt{1000110},
\texttt{1100001},
\texttt{00011000}\}$ are not imaged in $w$.

Let $I=\{\varepsilon,
\texttt{0},
\texttt{00},
\texttt{000},
\texttt{0000},
\texttt{00001},
\texttt{000011},
\texttt{0001},
\texttt{00011},
\texttt{000110},
\texttt{0001100},
\texttt{001},
\texttt{0011},\\
\texttt{00110},
\texttt{001100},
\texttt{0011000},
\texttt{01},
\texttt{011},
\texttt{0110},
\texttt{01100},
\texttt{011000},
\texttt{1},
\texttt{10},
\texttt{100},
\texttt{1000},
\texttt{10000},
\texttt{100001},\\
\texttt{10001},
\texttt{100011},
\texttt{11},
\texttt{110},
\texttt{1100},
\texttt{11000},
\texttt{110000},
\texttt{110001},
\texttt{1100011}\}$ be the set of words that avoid $T'$.
Then every imaged factor of $w$ is in $I$. Thus $w$ contains at most $|I|=36$ imaged factors.

Finally,~\Cref{no35} can be verified with an exhaustive search. If there exists a right-infinite word $z$ with at most 35 imaged factors,
then $z$ contains \texttt{001} or \texttt{110}, since otherwise $z$ would contain $(\texttt{01})^\omega$, $\texttt{0}^\omega$, or $\texttt{1}^\omega$ as a suffix.
So without loss of generality, we assume that \texttt{001} is a prefix of $z$.
Then, we perform depth-first exploration of the binary words with prefix \texttt{001} using this straightforward procedure:
we backtrack if the current word $w$ contains at least $36$ imaged factors, and we try to extend $w$ otherwise.
Since this backtracking program ends, there are no infinite binary words with at most 35 imaged factors.
To count the number of imaged factors faster, we pre-compute the set $A$ of all the factors of $w$ and the set $S_{uu}$ (resp. $C_{uuu}$)
of the factors $u$ of $w$ such that $uu$ (resp. $uuu$) is also a factor of $w$.
Now, if $f$ is a factor of $w$ containing the factor \texttt{00} and we look for an image $m(f)$ that is also a factor of $w$, then we know that $m(\texttt{0})\in S_{uu}$.
Since $S_{uu}$ is much smaller than $A$, the search for $m$ is faster.
We provide in the associated git repository a program that implements this exhaustive search~\cite{giteRepo(2025)} and runs in a few minutes on an ordinary laptop.

\section{Concluding remarks}
In this paper, we have considered infinite words that aim at limiting the amount of imaged factors.
This rules out pure morphic words (i.e., fixed points of an admissible morphism),
since every factor of a pure morphic word is imaged by its defining morphism.
However, these words can be morphic.
Indeed, our words are constructed as morphic images of arbitrary ternary $\tfrac74^+$-free words and
Dejean~\cite{Dejean(1972)} constructed a pure morphic ternary $\tfrac74^+$-free word.

One may consider a more extreme version of the notion in this paper by allowing
$f$ to be any binary word, not necessarily a factor of the considered infinite word $w$ itself.
We notice that every binary word, and in particular every factor of $w$, is the complement
of a binary word. So again, to avoid trivialities, we have to put another constraint on the morphisms.
The \emph{complentary morphism} is $\texttt{0}\rightarrow\texttt{1}$; $\texttt{1}\rightarrow\texttt{0}$.
A morphism is \emph{neat} if it is non-erasing, distinct from the identity, and distinct from the complementary morphism.
Given an infinite binary word $w$, we say that a finite binary word $b$ is \emph{realized}
in $w$ if there exists a neat morphism $m$ such that $m(b)$ is a factor of $w$.
Thus, $b$ can be realized in $w$ even if $b$ is not a factor of $w$.
We also notice that if $b$ is realized in $w$, then every factor of $b$ is realized in $w$.

Consider any infinite binary word $w$ that contains no squares other than
\texttt{00}, \texttt{11}, and \texttt{0101}~\cite{Fraenkel(1995),Gabric(2021)}. In particular, $w$ contains no 4-power.
We can quickly check that there exist only finitely many words that are realized in $w$. 
Consider a square $u$ with period at least 2, that is, $u=vv$ with $p=|v|\ge2$.
Suppose that $u$ is realized in $w$, that is, there exists a neat morphism $m$
such that $m(u)$ is a factor of $w$. First, we rule out the case where $v$ does not contain
both \texttt{0} and \texttt{1}, since $u$ and $m(u)$ would contain a 4-power.
Similarly, we rule out the case where $m$ is of the form 
$\texttt{0}\rightarrow x; \texttt{1}\rightarrow x$, since $m(u)=x^{2p}$ would contain~$x^4$.
Since $m$ is neat, this implies that $|m(\texttt{0})|+|m(\texttt{1})|\ge3$
and thus $|m(v)|\ge3$. Therefore, $w$ contains $m(u)=m(v)m(v)$, which is a square with period at least~3.
This contradiction shows that words realized in $w$ do not contain a square with period at least 2.
It is well known that the length of a binary word with no square with period at least 2 is at most 18.
Therefore, there are finitely many words realized in $w$.

Similarly to the results in this paper, we can ask for infinite binary words
that optimally bound the length and the number of realized words.

\nocite{*}
%\bibliographystyle{abbrvnat}
% use the following instead if you encounter problems 
% \bibliographystyle{alpha}
\bibliographystyle{abbrv}
\bibliography{bib}
\label{sec:biblio}

\end{document}